\documentclass[11pt]{amsart}
\usepackage{amsxtra,amssymb,amsmath,amsthm,amsbsy,amscd}
\usepackage{pb-diagram}
\usepackage{enumerate}
\usepackage{latexsym}
\usepackage[Glenn]{fncychap}
\usepackage[all]{xy}

\theoremstyle{plain}

\numberwithin{equation}{section}

\newtheorem{theorem}{Theorem}[section]
\newtheorem{lemma}[theorem]{Lemma}
\newtheorem{definition-lemma}[theorem]{Definition-Lemma}
\newtheorem{proposition}[theorem]{Proposition}
\newtheorem{corollary}[theorem]{Corollary}
\newtheorem{definition}[theorem]{Definition}
\newtheorem{remark}[theorem]{Remark}

\theoremstyle{definition}
\newtheorem{example}{Example}[section]

\newcommand{\Ker}        {{\mathrm {ker}}}
\newcommand{\ad}         {{\mathrm {ad}}}

\newcommand{\Ad}         {{\mathrm {Ad}}}

\newcommand{\pr}         {\mathrm{pr}}
\newcommand{\Ann}        {\mathrm{Ann}}

\newcommand{\w}         {\omega}
\newcommand{\opk}       {\oplus_{(k)}}
\newcommand{\TM}            {T^*M}
\newcommand{\TQ}            {T^*Q}
\newcommand{\gpd}       {\mathcal{G}}
\newcommand{\G}         {\mathbb{G}}
\newcommand{\LG}            {\mathfrak{g}}

\newcommand{\inv}         {{\mathrm{inv}}}

\newcommand{\ca}{[\![}
\newcommand{\cc}{]\!]}

\newcommand{\Lie}        {\mathcal L}

\begin{document}
\title[]
{Poly-symplectic groupoids and poly-Poisson structures}

\author[]{Nicolas Martinez}

\address{Instituto de Matem\'atica Pura e Aplicada,
Estrada Dona Castorina 110, Rio de Janeiro, 22460-320, Brasil }
\email{nicolasm@impa.br}

\address{}
\email{}


\date{}

\begin{abstract}
We introduce poly-symplectic groupoids, which are natural extensions
of symplectic groupoids to the context of poly-symplectic geometry,
and define poly-Poisson structures as their infinitesimal
counterparts. We present equivalent descriptions of poly-Poisson
structures, including one related with AV-Dirac structures. We also
discuss symmetries and reduction in the setting of poly-symplectic
groupoids and poly-Poisson structures, and use our viewpoint to
revisit results and develop new aspects of the theory initiated in
\cite{IMV}.
\end{abstract}

\maketitle


\tableofcontents
\section{Introduction}

Poly-symplectic structures arise in the geometric formulation of
Classical Field Theories in the same way that symplectic structures
appear in the Hamiltonian formalism of classical mechanics
\cite{Gu}. More precisely, poly-symplectic structures are
$\mathbb{R}^k$-valued 2-form, which are closed and satisfy a
nondegeneracy condition, in such a way that they concide with usual
symplectic forms when $k=1$. Poly-symplectic geometry has been
studied in recent years by several authors, including
\cite{Aw,AwGo,LMS,LV,No1}; see
also \cite{FoGo,Ka,LMerS,McNo,R-RSV} for further connections with physics. \\

In recent work \cite{IMV}, D. Iglesias J.C. Marrero and M. Vaquero
introduced a generalization of Poisson structure by considering
the inverse structures of poly-symplectic forms, analogous to the
way Poisson structures are defined from symplectic forms.
In this paper, we give a new viewpoint and study new aspects of the
work in \cite{IMV} by considering a slight variation of their
definition of poly-Poisson structure. Our definiton relies on the
relationship between symplectic groupoids  and Poisson manifolds
\cite{We,CDW}, but now in the setting of {\em poly-symplectic
groupoids}, which are natural extensions of symplectic groupoids to
poly-symplectic geometry.\\

Similarly to symplectic groupoids, {\em poly-symplectic groupoids}
are defined by a poly-symplectic form on a Lie groupoid satisfying a
compatibility condition, which says that the poly-symplectic form is
{\it multiplicative} (in the sense of \eqref{eq:mult} below). One of
the main properties of symplectic groupoids is that they are the
global versions of Poisson structures (see \cite{We,CDW}), that is,
the manifold of objects of a symplectic groupoid is endowed with a
Poisson structure whose corresponding Lie algebroid is isomorphic to
the Lie algebroid of the groupoid. Moreover, the Poisson structure
is uniquely determined by the condition that the target map is a
Poisson morphism. Starting with a poly-symplectic groupoid, the
corresponding infinitesimal geometric structure is what we identify
and call {\it poly-Poisson structure}. In other words, the
poly-Poisson structures we introduce here relate to poly-symplectic
groupoids exactly in the same way that Poisson structures relate to
symplectic groupoids. A similar idea in the context of
multi-symplectic geometry (see \cite{CIL,CIL1}) is studied in \cite{BCI}.\\

The notion of $k$-poly-Poisson structure arising in this way is
slightly less general than the one given in \cite{IMV}, but contains
the essential examples of the theory. Moreover, for $k=1$, the
notion agrees with ordinary Poisson structures (in contrast with the
more general definition of \cite{IMV}). From our viewpoint to
poly-Poisson structures, we will revisit some results in \cite{IMV}
and extend known facts about Poisson structures, e.g., concerning
their underlying Lie algebroids and foliations. Also, following the
description of Poisson structures as particular cases of Dirac
structures \cite{Co}, we discuss an analogous picture for
poly-Poisson structures.  In this case, however, Dirac structures
are not enough, and we must consider AV-Courant algebroids and a
suitable extension of {AV-Dirac structures}, as in \cite{LiB}.
\\

Poly-symplectic manifolds $M$ equipped with symmetries given by a
Lie group $\G$ induce, under suitable regularity conditions, a
quotient poly-Poisson structure on the manifold $M/\G$. In order to
find poly-symplectic groupoids integrating such quotients, we need
to discuss some aspects of hamiltonian actions and Marsden-Weinstein
reduction in poly-symplectic geometry, see e.g. \cite{Gu,MR-RSV}.
This allows us to
extend some constructions in \cite{MiWe,FOR} and \cite{BC2}, and
show that the symmetries $\G$ of an integrable poly-Poisson manifold
can be lifted to hamiltonian symmetries of its integrating
(source-simply-connected) poly-symplectic groupoid, and that its
poly-symplectic reduction at level zero is a poly-symplectic
groupoid integrating the quotient poly-Poisson structure on $M/\G$.
\\

There are several aspects of the approach to higher Poisson
structures considered in this paper that we plan to pursue in future
work, including the study of normal forms (see \cite{Aw,Ma} and the
more recent work in \cite{FoGo}), the geometry of the corresponding
higher versions of Dirac structures, and the potential connections
with Field Theory.



This paper is organized as follows: In Section \ref{sec:poly}  we
introduce  poly-symplectic groupoids.
The key result of this section, which generalizes \cite[Prop.
4.1]{BCI}, is Proposition \ref{Prop:nondeg. IM-form}, where we
obtain the relation between global and infinitesimal objects.
Poly-Poisson structures are defined in Section \ref{sec:ppois},
where we discuss their Lie algebroid structure, the underlying
foliation, together with their relation with poly-symplectic
groupoids via integration.
Poly-Poisson structures are illustrated with some examples from
\cite{IMV}. At the end we give a different way to describe
poly-Poisson structures related to AV-Dirac structures \cite{LiB}.
Section 4 is devoted to the study of symmetries of poly-Poisson
structures and hamiltonian actions on poly-symplectic manifolds, see
Theorem \ref{Thm:PPred} and Prop.~\ref{Prp:Reduction}. Finally,
applying the hamiltonian
reduction, we describe integrations of quotients of an integrable poly-Poisson manifold. \\

\noindent{\em Acknowledgments}: I would like to thanks Henrique Bursztyn for 
enlightening and helpful discussions and Juan Carlos Marrero and David Iglesias for 
their comments on this notes. This work was supported by a fellowship from CNPq.\\

\noindent {\it Notation}: Lie algebroids will be denoted by $A\to
M$, with anchor map $\rho:A\to TM$ and bracket $[\cdot, \cdot ]$.
For a Lie groupoid $\gpd$ over $M$, the source and target maps will
be $s,t:\gpd \to M$, $\epsilon: M\to \mathcal{G}$ denotes the unit
map, $\mathrm{inv}: \mathcal{G}\to \mathcal{G}$ is the inversion
map, and the groupoid multiplication is $m:\mathcal{G}_{(2)}\to
\mathcal{G}$, where the space of composable arrows is
$\mathcal{G}_{(2)}:=\mathcal{G}\times_{s,t} \mathcal{G}=\{(g,h)\in
\gpd\times \gpd |t(h)=s(g)\}$. The right and left translation on the
groupoid are $R_g,L_g$, respectively, for $g\in \gpd$.

For a vector space $V$, we will denote by $\oplus_{(k)}V$ the
$k$-fold direct sum of $V$, or equivalently, the space $V\otimes
\mathbb{R}^k$. On vector spaces we will use two different
annihilator spaces. For a vector subspaces $W$ of a vector space
$V$, we will denote by $\Ann(W)$ the space of elements on $V^*$
vanishing on $W$. For any subspace $S$ of $\opk V^*$, $S^o$ stands
for the space of elements on $V$ which annihilate the elements of
$S$, i.e $S^\circ=\{v\in V| \alpha(v)=0 $ for all $\alpha \in S\}$.
This notation will be used, more generally, for vector bundles $E\to
M$ rather than vector spaces.


The coadjoint action $\Ad^*:\G \to \mathrm{End}(\LG^*)$ of a Lie
group $\G$ on the dual of its Lie algebra $\LG$ induces a diagonal
coadjoint action of $\G$ on the product $\LG_{(k)}^*$, and we keep
the notation $\Ad^*$ to this action, i.e.,
$\Ad_g^*(\zeta_1,\dots,\zeta_k)=(\Ad_g^*\zeta_1,\dots,\Ad_g^*\zeta_k)$.

\section{Poly-symplectic groupoids}\label{sec:poly}

In this section we will recall the concept of poly-symplectic
manifold (see e.g. \cite{Gu,IMV,Aw}) and introduce poly-symplectic
groupoids, which will guide us towards poly-Poisson structures.

\subsection{Poly-symplectic structures}\label{subsec:polysymp}

A {\it $k$-poly-symplectic  form} on a manifold $M$ is a an
$\mathbb{R}^k$-valued differential form $\w\in
\Omega^2(M,\mathbb{R}^k)$ which is closed and nondegenerate, in the
sense that the induced bundle map
\begin{equation}\label{eq:wmap}
\w^\flat : TM \to T^*M\otimes \mathbb{R}^k
\end{equation}
is injective ($\Ker(\w)=\{0\}$). Writing $\w$ in terms of its
components, $\w=(\w_1,\ldots,\w_k)$, it is poly-symplectic if and
only if each $\w_j\in \Omega^2(M)$ is closed and
$$
\cap_{j=1}^k \Ker (\w_j)=\{0\}.
$$


One way to obtain examples of poly-symplectic structures is the
following. Let $M$ be a manifold endowed with $k$ surjective,
submersion maps $p_j:M\to M_j$, such that $\cap_{j=1}^k \Ker
(dp_j)=\{0\}$. If each $M_j$ is equipped with a
$l_j$-poly-symplectic form $\w_j$, then
$$
\w=(p_1^*\w_1,\dots ,p_k^*\w_k)
$$
is an $l$-poly-symplectic form on $M$, where $l=l_1+\ldots+l_k$. In particular, if $(M_j,\w_j)$ is an
$l_j$-poly-symplectic manifold, $j=1,\ldots,k$, this construction
endows $M:=M_1\times \dots \times M_k$ with an $l$-poly-symplectic
structure, for $l=l_1+\ldots+l_k$. This shows that the product of
$k$ symplectic manifolds naturally carries a $k$-poly-symplectic
structure.

The following is a particular case of interest in classical field
theory \cite{Gu}:

\begin{example}\label{Ex:polysymplectic} ($k$-covelocities on a manifold)
Recall that any cotangent bundle $T^*Q$ has a canonical symplectic
form $\w_{can}$. The manifold of $k$-covelocities is the Whitney sum 
$$
\opk \TQ=T^*Q\oplus \overset{(k}{\cdots} \oplus T^*Q,
$$
which is equipped with the natural projections $\pr_j:\opk \TQ \to
T^*Q$. It is clear that $\cap_{j=1}^k \Ker (d\pr_j) =\{0\}$, and
$$
\w:=(\pr_1^*\w_{can},\dots ,\pr_k^*\w_{can})\in \Omega^2(\opk
\TQ,\mathbb{R}^k)
$$
is a $k$-poly-symplectic form.
\end{example}

Other examples of poly-symplectic structures are discussed e.g. in
\cite{Gu,IMV,No}.

\subsection{Multiplicative forms and poly-symplectic groupoids}

We now consider poly-symplectic structures on Lie groupoids. Let
$\mathcal{G}$ be a Lie groupoid over $M$.

A differential form $\theta\in \Omega^r(\mathcal{G})$ is called {\it
multiplicative} if it satisfies
\begin{equation}\label{eq:mult}
m^*\theta = \pr_1^*\theta + \pr_2^*\theta,
\end{equation}
where $\pr_i:\mathcal{G}\times_{s,t} \mathcal{G}\to \mathcal{G}$ are
the projection maps. Note that condition \eqref{eq:mult} still makes
sense for $\mathbb{R}^k$-valued forms
$\theta=(\theta_1,\ldots,\theta_k)$, and it simply says that each
component $\theta_i$ is multiplicative.

Recall that a {\it symplectic groupoid} is a Lie groupoid
$\mathcal{G}\rightrightarrows M$ endowed with a multiplicative
symplectic form $\omega\in \Omega^2(\mathcal{G})$, see e.g.
\cite{CDW,We}. A direct generalization leads to

\begin{definition}
A {\it $k$-poly-symplectic  groupoid} is a Lie groupoid
$\mathcal{G}\rightrightarrows M$ together with a $k$-poly-symplectic
form $\omega= (\omega_1,\ldots,\omega_k) \in
\Omega^2(\mathcal{G},\mathbb{R}^k)$ satisfying \eqref{eq:mult}. More
explicitly, each $\omega_j\in \Omega^2(\mathcal{G})$ is closed,
multiplicative, and $\cap_{j=1}^k \Ker(\omega_j)=\{0\}$.
\end{definition}

Suppose that $\mathcal{G}_j\rightrightarrows M_j$ are
$l_j$-poly-symplectic groupoids, $j=1,\ldots,k$. As discussed in
Section~\ref{subsec:polysymp}, we can verify that if a Lie groupoid
$\mathcal{G}$ is equipped with surjective submersions $p_j:
\mathcal{G}\to \mathcal{G}_j$, $j=1,\ldots,k$, which are {\it
groupoid morphisms} and satisfy $\cap_j \Ker(dp_j)=\{0\}$, then
$\omega=(p_1^*\omega_1,\ldots,p_k^*\omega) \in
\Omega^2(\mathcal{G},\mathbb{R}^k)$ makes $\mathcal{G}$ into an
$l$-poly-symplectic groupoid, for $l=l_1+\ldots+l_k$. Here we use
the fact that the pullback of a multiplicative form by a groupoid
morphism is again multiplicative. In particular, we have:

\begin{proposition}\label{prop:prod}
The direct product of symplectic groupoids
$(\mathcal{G}_j,\omega_j)$, $j=1,\ldots,k$, naturally carries a
multiplicative $k$-poly-symplectic structure given by
$$
\omega=(\pr_1^*\omega_1,\ldots,\pr_k^*\omega_k),
$$
where $\pr_j: \mathcal{G}_1\times\ldots\times \mathcal{G}_k\to
\mathcal{G}_j$ is the natural projection.
\end{proposition}

More conceptually, multiplicative poly-symplectic forms are very
special cases of multiplicative forms with values in
representations, as in \cite{CSS}. Given a Lie groupoid
$\mathcal{G}\rightrightarrows M$ and a vector bundle $E\to M$,
consider the pullback bundle $t^*E\to \mathcal{G}$. An {\it
$E$-valued $r$-form} on $\mathcal{G}$ is an element $\theta \in
\Omega^r(\mathcal{G},t^*E)$. If $E$ is a representation of
$\mathcal{G}$ (see \cite{Mc}), we say that $\theta \in
\Omega^r(\mathcal{G},t^*E)$ is {\it multiplicative} if for all
composable arrows $(g,h)\in \mathcal{G}\times_{s,t}\mathcal{G}$ we
have
\begin{equation}\label{Eq:E-mult.}
(m^*\theta)_{(g,h)}=\pr_1^*\theta+g\cdot(\pr_2^*\theta),
\end{equation}
where $m,\pr_1,\pr_2$ are as in \eqref{eq:mult}. It is clear that
for the trivial bundle $E=\mathbb{R}^k\times M$, equipped with the
trivial representation, this recovers the notion of multiplicative
$\mathbb{R}^k$-valued forms previously discussed.

For later use, we observe the $E$-valued version of the equations in
\cite[Lemma~3.1(i)]{BCWZ}:
\begin{lemma}\label{Lem:E-mult.prop.}
If $\theta\in \Omega^{k}(\mathcal{G},t^*E)$ is multiplicative then
\begin{equation}\label{Eq:E-mult.prop.}
\epsilon^*\theta=0,\;\;  \mbox{and } \;\; \theta_g=-g\cdot
(\mathrm{inv}^*\theta_{\mathrm{inv}(g)})
\end{equation}
for all $g\in \mathcal{G}$.
\end{lemma}

\begin{proof}
Define the map $(Id\times \inv)(g):=(g,g^{-1})$ from $\gpd$ to $\gpd
_{(2)}$. If we apply the pull-back of $(Id\times \inv)$ to Equation
\eqref{Eq:E-mult.} and recall that $\epsilon \circ t = m\circ (Id
\times \inv)$, we obtain:
\begin{align*}
t^*\epsilon^ *\theta_{\epsilon(t(g))}&=(Id\times
\inv)^*(m^*\theta)_{(g,g^{-1})}=
\theta_ g+(Id\times \inv)^*(g\cdot (pr_2^*\theta_{g^{-1}}))\\
&=\theta_ g+g\cdot ((Id\times \inv)^*pr_2^*\theta_{g^{-1}}).
\end{align*}
Therefore
\begin{equation}\label{eq:lem}
t^*\epsilon^ *\theta_{\epsilon(t(g))}=\theta_
g+g\cdot(\inv^*\theta_{g^{-1}}).
\end{equation}
 If in particular we fix
$g=\epsilon(m)$ for some $m\in M$ and take the pull-back by the unit
map in \eqref{eq:lem}, we conclude that $\epsilon^*\theta=0$. Using
this identity and \eqref{eq:lem}, it follows that $\theta_
g+g\cdot(\inv^*\theta_{g^{-1}})=0$.
\end{proof}

\subsection{Infinitesimal data of poly-symplectic groupoids}

It is well known that Poisson structures are the infinitesimal
counterparts of symplectic groupoids, see e.g. \cite{CDW,We}. We will
now discuss the infinitesimal counterpart of poly-symplectic
groupoids, in the spirit of \cite{BCI}, which leads to a
generalization of Poisson structures in poly-symplectic geometry.

Let $A\to M$ denote the Lie algebroid of a Lie groupoid $\mathcal{G}
\rightrightarrows M$, with anchor $\rho: A\to TM$ and bracket
$[\cdot,\cdot]$ on $\Gamma(A)$. Recall from \cite{AC,BC1,BCWZ} that
a closed multiplicative $r$-form $\theta$ on $\mathcal{G}$ is
infinitesimally described by a bundle map (over the identity)
$$
\mu: A \to \wedge^{r-1}T^*M,
$$
satisfying the conditions
\begin{align}\label{Eq:a-IM-mform}
i_{\rho(u)}\mu(v)&=-i_{\rho(v)}\mu(u), \;\;\; \forall\, u,v \in A\\ 
\mu([u,v])&=\Lie_{\rho(u)}\mu(v)-i_{\rho(v)}d\mu(u),\;\;\; \forall\,
u,v \in \Gamma(A). \label{Eq:b-IM-mform}
\end{align}
The map $\mu$ is related to $\theta$ via
\begin{equation}\label{eq:rel1}
i_{u^R}\theta = t^*(\mu(u)),
\end{equation}
for $u\in \Gamma(A)$, where $u^R$ denotes the right-invariant vector
field on $\mathcal{G}$ defined by $u$. For source-simply-connected
Lie groupoids, $\mu$ and $\theta$ completely determine one another.

It follows that a closed multiplicative $\mathbb{R}^k$-valued 2-form
$\omega =(\omega_1,\ldots,\omega_k) \in
\Omega^2(\mathcal{G},\mathbb{R}^k)$ infinitesimally corresponds to a
bundle map
\begin{equation}\label{eq:mu_k}
\mu = (\mu_1,\ldots,\mu_k): A\to \oplus_{(k)}T^*M
\end{equation}
satisfying the same equations \eqref{Eq:a-IM-mform} an
\eqref{Eq:b-IM-mform}, which simply means the equations are
satisfied componentwise, i.e., each $\mu_j: A\to T^*M$ is a closed
IM 2-form. For the complete infinitesimal description of a
multiplicative poly-symplectic form, it remains to express the
non-degeneracy condition $\cap_{j=1}^k\Ker(\omega_j)=\{0\}$ in terms
of the map $\mu$ in \eqref{eq:mu_k}. We will do that in the more
general framework of multiplicative forms on $\mathcal{G}$ with
values in representations $E\to M$.

The infinitesimal version of multiplicative $E$-valued $r$-forms on
a Lie groupoid $\mathcal{G}$ was studied in \cite{CSS}, where it is
proven that (under the usual source-simply-connectedness condition
on $\mathcal{G}$) such forms $\theta$ are in 1-1 correspondence with
pairs of maps $(D,\mu)$,

\begin{center}\begin{tabular}{lr}
$D:\Gamma(A)\to \Omega^{r}(M,E),$&$\mu:A\to \wedge^{r-1}T^*M\otimes
E$,
\end{tabular}\end{center}
satisfying suitable conditions (that we will not need explicitly),
see \cite[Sec.~2.2]{CSS}. We will only need the following facts
about the infinitesimal data $(D,\mu)$. First, the relation between
the bundle map $\mu$ and the multiplicative $E$-valued form $\theta$
is a direct generalization of that in \eqref{eq:rel1}: indeed, using
\cite[Eqs. (3.1)-(3.3)]{BCWZ}, it follows that the second equation
of \cite[(2.4)]{CSS} is equivalent to
\begin{equation}\label{eq:muE}
i_{u^R}\theta =  t^*(\mu(u)).
\end{equation}
Second, when $E=\mathbb{R}^k$ is the {\it trivial} representation
and the multiplicative form $\theta$ is {\it closed}, then $D$ is
determined by $\mu$, in fact $D=d\mu$ (see \cite{BC1}); so in this
case one only needs $\mu$ for the infinitesimal description of
$\theta$.

We say that an $r$-form $\theta \in \Omega^r(\mathcal{G},t^*E)$  is
{\it non-degenerate} when the map
$$
\theta^\flat: T\mathcal{G} \to \wedge^{r-1}T^*\mathcal{G}\otimes
t^*E, \;\;\; X\mapsto i_X\theta
$$
has trivial kernel. When $\theta$ is multiplicative, we have the
following infinitesimal description of this property.

\begin{proposition} \label{Prop:nondeg. IM-form}
Consider $\theta\in \Omega^r(\mathcal{G},t^*E)$ a multiplicative
$E$-valued $r$-form on a Lie groupoid $\mathcal{G}$, and let $\mu:
A\to \wedge^{r-1}T^*M \otimes E$ be such that \eqref{eq:muE} holds.
Then $\theta$ is nondegenerate if and only if
\begin{equation}\label{eq:munondeg}
\Ker(\mu)=\{0\},\;\;  \mbox{ and } \;\; (\mathrm{Im}(\mu))^\circ=\{0\},
\end{equation}
where $(\mathrm{Im} (\mu))^\circ=\{X\in TM\,|\, i_X\mu(u)=0$ for all $u\in A\}$.
\end{proposition}
\begin{proof}
The  proof uses the relation \eqref{eq:muE} and follows the same
idea of \cite[Prop.~4.1]{BCI}. We recall the details for the
reader's convenience.

First we suppose that conditions (\ref{eq:munondeg}) hold for $\mu$
and take $X\in T_g\gpd$ in the kernel of the multiplicative form. We
get that $dtX=0$ because $i_Xt^*(\mu(u))=0$ for all $u\in A$ (from
\eqref{eq:muE}), hence $X$ is tangent to the $t$-fibers, which
implies the existence of $v\in A$ for which
$X=v^L_g=d_g\inv(v^R_g)$. As consequence of the second equation in
\eqref{Eq:E-mult.prop.} and \eqref{eq:muE}, we see that
$-g\cdot(s^*(\mu(v)))=i_{v_g^L}\theta_g=i_X\theta_g=0$ for any $g\in \mathcal{G}$, hence
$s^*(\mu(v))=0$. This shows that $v\in \Ker(\mu)$, therefore
$X=v_g^L=0$.

For the other direction, let $u\in \Ker(\mu)$. Then
$i_{u^R}\theta=t^*(\mu(u))=0$, which implies $u^R=0$ by
nondegeneracy of the form, thus the first condition in
\eqref{eq:munondeg} holds. Now fixing $X\in (\mathrm{Im}
(\mu))^\circ_m$ for $m\in M$, \eqref{eq:muE} implies that
$i_ui_X\theta=0$ for all $u\in A_m$. The splitting $T_m\gpd
=T_mM\oplus A_m$ allows us to write $Z_j=X_j+u_j\in T_m\gpd,
j=1,\dots ,r-1$, and the multilinearity of $\theta$ implies that
\[
i_{Z_{r-1}}\dots i_{Z_1}i_X\theta=i_{X_{r-1}}\dots i_{X_1}i_X\theta,
\]
because the other terms vanish from the fact that $i_ui_X\theta=0$
for all $u\in A_m$. Now the first condition in \eqref{Eq:E-mult.}
implies that $i_{Z_{r-1}}\dots i_{Z_1}i_X\theta=0$ for all $Z_j\in
T_m\gpd$, hence $X=0$.
\end{proof}

For the trivial representation $E=M \times \mathbb{R}$ and forms of
arbitrary degree $r$, Proposition~ \ref{Prop:nondeg. IM-form}
recovers \cite[Prop.~4.1]{BCI}. For the trivial representation $E=M
\times \mathbb{R}^k$ and $r=2$, we obtain the infinitesimal
description of multiplicative $k$-poly-symplectic forms.

\begin{corollary}\label{cor:polymu}
Let $\mathcal{G}\rightrightarrows M$ be a source-simply-connected
groupoid. Then there is a one-to-one correspondence between
multiplicative poly-symplectic forms $\omega\in
\Omega^2(\mathcal{G},\mathbb{R}^k)$ and bundle maps $\mu: A\to
\oplus_{(k)}T^*M$ satisfying \eqref{Eq:a-IM-mform},
\eqref{Eq:b-IM-mform} and \eqref{eq:munondeg} via the relation
$i_{u^R}\omega=t^*(\mu(u))$, for all $u\in \Gamma(A)$.
\end{corollary}

Given a Lie algebroid $A\to M$, we see from
Corollary~\ref{cor:polymu} that bundle maps $\mu: A\to
\oplus_{(k)}T^*M$ satisfying \eqref{Eq:a-IM-mform},
\eqref{Eq:b-IM-mform} and \eqref{eq:munondeg} are the infinitesimal
counterparts of multiplicative poly-symplectic forms on Lie
groupoids. So we refer to these objects as {\it IM poly-symplectic
forms}, where ``IM'' stands for ``infinitesimally multiplicative''.
We say that two IM poly-symplectic forms $\mu: A\to
\oplus_{(k)}T^*M$ and $\mu': A'\to \oplus_{(k)}T^*M$ are {\it
equivalent} if there is a Lie algebroid isomorphism $\varphi: A\to
A'$ such that $\mu = \mu'\circ \varphi$. Under the equivalence in
Corollary~\ref{cor:polymu}, they correspond to isomorphic
poly-symplectic groupoids.

We will now use the infinitesimal geometry of poly-symplectic
groupoids described in Corollary~\ref{cor:polymu} to provide a new
viewpoint to \cite{IMV}.

\section{Poly-Poisson structures}\label{sec:ppois}

\subsection{Definition}\label{sebsec:def}
The notion of {\it poly-Poisson structure} that we now introduce is
a slight modification of that in \cite{IMV}.

\begin{definition}\label{Def:PP}
A {\em $k$-poly-Poisson  structure} on a manifold $M$ is a pair
$(S,P)$, where $S\to M$ is a vector subbundle of $\opk \TM$ and
$P:S\to TM$ is a vector-bundle morphism (over the identity) such
that
\begin{itemize}
\item[(i)] $i_{P(\bar{\eta})} \bar{\eta}=0$, for all $\bar{\eta}\in
S$,
\item[(ii)] $S^\circ= \{X\in TM| i_X\bar{\eta}=0, \,\forall\,\bar{\eta}\in
S\}=\{0\}$,
\item[(iii)] the space of section $\Gamma(S)$ is closed under the
bracket
\begin{equation}
\label{Def:bracketPP} \lfloor \bar{\eta},\bar{\gamma}
\rfloor:=\Lie_{P(\bar{\eta})}\bar{\gamma}-
\Lie_{P(\bar{\gamma})}\bar{\eta}+d(i_{P(\bar{\gamma})}\bar{\eta})=\Lie_{P(\bar{\eta})}\bar{\gamma}-
i_{P(\bar{\gamma})}d\bar{\eta}, \; \mbox{ for } \;
\bar{\gamma},\bar{\eta}\in \Gamma(S),
\end{equation}
and the restriction of this bracket to $\Gamma(S)$ satisfies the
Jacobi identity.
\end{itemize}
\end{definition}

We will call the triple $(M,S,P)$ a {\it $k$-poly-Poisson manifold}.

We observe that the bracket \eqref{Def:bracketPP} is skew-symmetric
(by condition (i)) and satisfies the Leibniz rule:
$$
\lfloor \bar{\eta}, f\bar{\gamma} \rfloor=f\lfloor \bar{\eta},
\bar{\gamma}\rfloor + (\Lie_{P(\bar{\eta})}f)\bar{\gamma},
$$
for all $\bar{\eta}, \bar{\gamma} \in \Gamma(S)$ and $f\in
C^\infty(M)$. It follows that, for a poly-Poisson manifold
$(M,S,P)$, the vector bundle $S\to M$ is a {\it Lie algebroid} with
bracket \eqref{Def:bracketPP} and anchor map $P: S\to TM$. Since for
any Lie algebroid the anchor map preserves Lie brackets, we have
that
\begin{equation}\label{eq:brk}
P(\lfloor \bar{\eta}, \bar{\gamma} \rfloor) =
[P(\bar{\eta}),P(\bar{\gamma})], \;\; \forall \; \bar{\eta},
\bar{\gamma} \in \Gamma(S).
\end{equation}

\begin{remark}\label{rem:integcond}
In $(iii)$ of Def.~\ref{Def:PP}, assuming that $\Gamma(S)$ is closed
under the bracket \eqref{Def:bracketPP}, we can replace the
condition on the Jacobi identity by the bracket-preserving property
\eqref{eq:brk}. Indeed, if \eqref{eq:brk} holds and for $\bar{\eta}, \bar{\lambda}, \bar{\gamma} \in \Gamma(S)$, then
\begin{align*}
\lfloor \lfloor \bar{\eta},\bar{\gamma}\rfloor, &
\bar{\lambda}\rfloor  +\lfloor \bar{\gamma}, \lfloor \bar{\eta}
,\bar{\lambda}\rfloor \rfloor =
\Lie_{[P(\bar{\eta}),P(\bar{\gamma})]}\bar{\lambda}
-i_{P(\bar{\lambda})}d\lfloor \bar{\eta},\bar{\lambda}\rfloor +
\Lie_{P(\bar{\gamma})}\lfloor \bar{\eta},\bar{\lambda}  \rfloor-
i_{[P(\bar{\eta}),P(\bar{\lambda})]}d\bar{\gamma}\\
&=\Lie_{P(\bar{\eta})}\Lie_{P(\bar{\gamma})}\bar{\lambda}-i_{P(\bar{\lambda})}\Lie_{P(\bar{\eta})}d\bar{\gamma}+ i_{P(\bar{\lambda})}\Lie_{P(\bar{\gamma})}d\bar{\eta}
-\Lie_{P(\bar{\gamma})}i_{P(\lambda)}d\bar{\eta}-i_{[P(\bar{\eta}),P(\bar{\lambda})]}d\bar{\gamma}\\
&=\Lie_{P(\bar{\eta})}\Lie_{P(\bar{\gamma})}\bar{\lambda}- \Lie_{P(\bar{\eta})}i_{P(\bar{\lambda})}d\bar{\gamma}
+i_{P(\bar{\lambda})}\Lie_{P(\bar{\gamma})}d\bar{\eta}- \Lie_{P(\bar{\gamma})}i_{P(\bar{\lambda})}d\bar{\eta}\\
&=\Lie_{P(\bar{\eta})}\lfloor \bar{\gamma},\bar{\lambda}\rfloor
-i_{[P(\bar{\gamma}),P(\bar{\lambda})]}d\bar{\eta}=\lfloor
\bar{\eta},\lfloor \bar{\gamma},\bar{\lambda}\rfloor \rfloor,
\end{align*}
where the second equality holds by $\Lie_{[X,Y]}=[\Lie_X,\Lie_Y]$
and the third results from Cartan's magic formula.

\end{remark}
It follows from this remark that condition (iii) in
Def.~\ref{Def:PP} is equivalent to
\begin{itemize}
\item[(iii)'] {\em the space of section $\Gamma(S)$ is closed under the
bracket \eqref{Def:bracketPP} and \eqref{eq:brk} holds.}
\end{itemize}

\begin{remark}[Comparison with \cite{IMV}]\label{rem:IMV}
The notion of poly-Poisson structure in Def.~\ref{Def:PP} is
slightly more restrictive than the notion  introduced by Iglesias,
Marrero and Vaquero in \cite[Def.~3.1]{IMV}. The difference is that
in \cite{IMV} our condition (ii) in Def.~\ref{Def:PP}, namely
$S^\circ=\{0\}$, is replaced by the following weaker requirement:
\begin{equation}\label{eq:iiweak}
\mathrm{Im}(P)\cap S^\circ = \{0\}.
\end{equation}
We will refer to such objects as {\em weak-poly-Poisson structures}.
\end{remark}

Let $(M_j,S_j,P_j)$, $j=1,2$, be $k$-poly-Poisson manifolds.

\begin{definition}\label{Def:mor}
A smooth map $f:M_1\to M_2$ is called a {\em poly-Poisson morphism}
if
\begin{enumerate}[a)]
\item $f^*\bar{\eta}\in S_1$ for all $\bar{\eta}\in S_2$,
\item for every $x\in M_1$ and $\bar{\eta}\in S_2|_{f(x)}$, $Tf|_x(P_1(Tf|_x^*\bar{\eta})) = P_2(\bar{\eta})$.
\end{enumerate}
\end{definition}

The following are basic examples of Def.~\ref{Def:PP}.

\begin{example}
For $k=1$, a $k$-poly-Poisson structure is simply a usual Poisson
structure. Indeed, if $S$ is subbundle of $T^*M$, condition (ii) in
Def.~\ref{Def:PP} shows that
$$
S=T^*M.
$$
(Note that this is not guaranteed by the weaker condition
\eqref{eq:iiweak}.) Condition (i) shows that $P:T^*M\to TM$ is of
the form $P=\pi^\sharp$ for a bivector field $\pi \in
\Gamma(\wedge^2TM)$, where $\pi^\sharp(\alpha)=i_\alpha\pi$.
Finally, condition (iii) amounts to the usual integrability
condition $[\pi,\pi]=0$ (i.e., the bracket on $C^\infty(M)$ given by
$(f,g)\mapsto \pi(df,dg)$ satisfies the Jacobi identity). The Lie
algebroid structure on $S=T^*M$ is the usual one for Poisson
manifolds \cite{Va}: the anchor is $\pi^\sharp$ and the
bracket $[\cdot,\cdot]$ on $\Omega^1(M)$ is the one such that
$[df,dg] = d(\pi(df,dg))$. The notion of morphism in
Def.~\ref{Def:mor} also recovers to the usual notion of Poisson
morphism.
\end{example}

\begin{example}\label{Ex:polysp}
Let $(M,\w)$ be a $k$-poly-symplectic  manifold, and consider the
injective bundle map $\w^\flat: TM \to \opk T^*M$. We define a
subbundle $S_\w$ of $\opk T^*M$ and a bundle map $P_\w:S\to TM$ as
follows:
\begin{equation}\label{Eq:S_wP_w}
\begin{tabular}{ccc}
$S_\w:=\mathrm{Im}(\w^\flat)$& and & $P_\w(i_X\w):=X\in TM$.\\
\end{tabular}
\end{equation}
See \cite[Prop.~2.3 and Example~3.3]{IMV}. Note that condition (ii)
in Def.~\ref{Def:PP} is equivalent to the non-degeneracy of $\w$.

Moreover, given $k$-poly-symplectic manifolds $(M_j,\w_j)$, $j=1,2$,
a diffeomorphism $f: M_1\to M_2$ preserves poly-Poisson structures
(as in Def.~\ref{Def:mor}) if and only if
$$
f^*\w_2=\w_1.
$$
\end{example}

\begin{example}\label{ex:PPtrivial}
Let $Q$ be a manifold. We can always regard it as a Poisson manifold
with the Poisson bracket that is identically zero. For each $k$, we
can also view $Q$ as a $k$-poly-Poisson manifold, and this can be
done in several ways. For example, $S_1=\oplus_{(k)}T^*Q$ and
$P_1=0$ define a poly-Poisson structure on $Q$, and the same is true
for $S_2=\{\alpha\oplus\ldots\oplus\alpha\,|\, \alpha \in T^*Q\}
\subset \oplus_{(k)}T^*Q$ and $P_2=0$, or
$S_3=\{\alpha\oplus0\oplus\ldots\oplus 0\,|\, \alpha \in T^*Q\}
\subset \oplus_{(k)}T^*Q$ and $P_3=0$.

Considering $\oplus_{(k)}T^*Q$ equipped with its poly-symplectic
structure (see Example~\ref{Ex:polysymplectic}), the natural
projection $\oplus_{(k)}T^*Q \to Q$ is a poly-Poisson map when $Q$
is equipped with either one of the poly-Poisson structures
$(S_i,P_i)$, for $i=1,2,3$.
\end{example}

\begin{remark}\label{rem:unique}
It is a well-known fact in Poisson geometry that $M$ is a Poisson
manifold and $f: M\to N$ is a surjective submersion, then there is
at most one Poisson structure on $N$ for which $f$ is a Poisson map.
Example~\ref{ex:PPtrivial} shows that this is not necessarily the
case for $k$-poly-Poisson structures, for $k\geq 2$.

On the other hand, let $M$ be a $k$-poly-Poisson manifold and
$f:M\to N$ be a surjective submersion. Then if $(S_1,P_1)$ and
$(S_2,P_2)$ are $k$-poly-Poisson structures on $N$ for which $f$ is
a poly-Poisson map and we know that $S_1=S_2$, then $P_1=P_2$.
\end{remark}

As explained in \cite[Example~3.8]{IMV}, the product of Poisson
manifolds carries a natural poly-Poisson structure.



\begin{example} \label{Ex:PPprod}
Let $(M_j,\pi_j)$, $j=1,\ldots,k$, be Poisson manifolds. Let
$M=M_1\times \ldots \times M_k$. Denote by $S_j$ the natural
inclusion of $T^*M_j$ into $T^*M$, and let $S\subset \opk T^*M$ be
defined by $S: = S_1\oplus \ldots \oplus S_k$. Consider the bundle
map $P: S\to TM$,
$$
P(\alpha_1,\ldots,\alpha_k) =
(\pi_1^\sharp(\alpha_1),\ldots,\pi_k^\sharp(\alpha_k)),
$$
where $\alpha_j\in S_j$. One may verify that $(M,S,P)$ is a
$k$-poly-Poisson manifold directly from the definition.

 In addition,
let $f_j:(M_j,\pi_j)\to (N_j,\Lambda_j)$ for $j=1,\dots,k$, be $k$
Poisson maps between the Poisson manifolds $M_j$ and $N_j$
respectively.  From the construction above we obtain
$k$-poly-Poisson structures $(S_M,P_M)$ and $(S_N,P_N)$  on the
product manifolds $M=\prod_{j=1}^kM_j$ and $N=\prod_{j=1}^kN_j$, and
denote by $pr^M_j: M\to M_j$, and $pr^N_j: N\to N_j$ the natural
projections. The Poisson maps $f_j$ induce a product map
$\bar{f}=(f_1,\dots,f_k):M\to N$ that, as a consequence of the
definition of the $k$-poly-Poisson manifold and the relations
$pr^N_j\circ \bar{f}=f_j\circ pr^M_j$, is a poly-Poisson map.

\end{example}

The next example is a particular case of the direct-sum of linear
Poisson structures treated in \cite[Example~3.9]{IMV}.

\begin{example}\label{Ex:Lalgb_k}
Let $\LG$ be a Lie algebra, and let
$$
\LG_{(k)}:=\LG\times \overset{(k}{\cdots} \times \LG,\;\;\;
\LG^*_{(k)}:=\LG^*\times \overset{(k}{\cdots} \times \LG^*.
$$
For $u\in \LG$, let $u_j \in \LG_{(k)}$ denote the element
$(0,\ldots,0,u,0,\ldots,0)$, with $u$ in the $j$-th entry.
 Since $\LG^*$ is equipped with its
Lie-Poisson structure, $\LG^*_{(k)}$ naturally carries a product
poly-Poisson structure, as in Example~\ref{Ex:PPprod}. More
important to us is the following {\it direct-sum} poly-Poisson
structure \cite{IMV} : over each
$\zeta=(\zeta_1,\dots ,\zeta_k)\in \LG^*_{(k)}$, we define
$$
S|_\zeta:=\{(u_1,\dots ,u_k)|u\in \LG\} \subseteq \opk T^*_\zeta
\LG^*_{(k)} \cong \opk \LG_{(k)},
$$
and the bundle map $P: S \to T \LG^*_{(k)}$,
$$
P_\zeta(u_1,\dots ,u_k):=(\ad_{u}^*\zeta_1,\dots ,\ad_{u}^*\zeta_k)
\in T_\zeta \LG^*_{(k)}\cong \LG^*_{(k)}.
$$
We remark that $S$ satisfies (ii) in Def.~\ref{Def:PP}, not just
\eqref{eq:iiweak}.


\end{example}

\subsection{Poly-Poisson structures and poly-symplectic groupoids}

We will now justify our definition of poly-Poisson structure in
Def.~\ref{Def:PP} in light of its relation with poly-symplectic
groupoids.

Let $(M,S,P)$ be a $k$-poly-Poisson manifold. We saw in
Section~\ref{sebsec:def} that the vector subbundle $S\subseteq \opk
T^*M$ is a Lie algebroid, with anchor $P:S \to TM$ and bracket
\eqref{Def:bracketPP}.

\begin{lemma}\label{lem:incl}
Let $\mu: S\hookrightarrow \opk T^*M$ be the inclusion. Then $\mu$
is an IM poly-symplectic form on the Lie algebroid $S\to M$, i.e.,
$\mu$ satisfies \eqref{Eq:a-IM-mform}, \eqref{Eq:b-IM-mform} and
\eqref{eq:munondeg}.

Conversely, any IM poly-symplectic form $\mu: A\to \opk T^*M$ is
equivalent to one coming from a $k$-poly-Poisson structure.
\end{lemma}

\begin{proof}
Note that \eqref{Eq:a-IM-mform} is just (i) in Def.~\ref{Def:PP},
while property \eqref{Eq:b-IM-mform} follows from (iii) in
Def.~\ref{Def:PP}. Since $\mu$ is an inclusion, $\Ker(\mu)=\{0\}$.
The second condition in \eqref{eq:munondeg} is (ii) in
Def.~\ref{Def:PP}.

On the other hand, given an IM poly-symplectic form $\mu: A\to \opk
T^*M$, we define $S=\mathrm{Im}(\mu)$. Note (from the first
condition in \eqref{eq:munondeg}) that $\mu$ is a vector-bundle
isomorphism onto $S$, and let $P: S\to TM$ be its inverse $S\to A$
composed with the anchor $A\to TM$. One may directly verify from
conditions \eqref{Eq:a-IM-mform}, \eqref{Eq:b-IM-mform} and
\eqref{eq:munondeg} that $S$ and $P$ define a $k$-poly-Poisson
structure, and that $\mu$ is equivalent to the inclusion
$S\hookrightarrow \opk T^*M$.

\end{proof}

In short, the lemma says that a $k$-poly-Poisson manifold $(M,S,P)$
endows $S$ with a Lie algebroid structure for which the inclusion
$S\hookrightarrow \opk T^*M$ is an IM poly-symplectic form, and that
any IM poly-symplectic form is equivalent to one of this type.

Following Corollary~\ref{cor:polymu}, we see that poly-Poisson
manifolds are the infinitesimal counterparts of poly-symplectic
groupoids, as explained by the next result. For a $k$-poly
symplectic groupoid $(\mathcal{G}\rightrightarrows M,\w)$, let $\mu:
A\to \oplus_{(k)}T^*M$ be the bundle map determined by $\omega$ as
in Cor.~\ref{cor:polymu}. Explicitly, using the natural
decomposition $T\mathcal{G}|_M = TM\oplus A$,
$$
\mu(u) = \omega^\flat(u)|_{\oplus_{(k)}TM},
$$
for $u\in A$.

\begin{theorem}[Integration of poly-Poisson structures]\label{Thm:integration}
If $(\mathcal{G}\rightrightarrows M,\w)$ is a k-poly-symplectic
groupoid, then there exists a unique $k$-poly-Poisson  structure
$(S,P)$ on $M$ such that $S=\mathrm{Im}(\mu)$ while $P$ is
determined by the fact that the target map $t:\mathcal{G}\to M$ is a
poly-Poisson morphism.

Conversely, let $(M,S,P)$ be a $k$-poly-Poisson  manifold and
$\mathcal{G}\rightrightarrows M$ be a source-simply-connected
groupoid integrating the Lie algebroid $S\to M$. Then there is a $\w
\in \Omega^2(\mathcal{G},\mathbb{R}^k)$, unique up to isomorphism,
making $\mathcal{G}$ into a poly-symplectic groupoid for which
$t:\mathcal{G}\to M$ is a poly-Poisson morphism.
\end{theorem}



We say that a poly-symplectic groupoid {\it integrates} a
poly-Poisosn structure if they are related as in the theorem. We
observe that this correspondence between source-simply-connected
poly-symplectic groupoids and poly-Poisson manifolds (with
integrable Lie algebroid) extends the well-known relationship
between symplectic groupoids and Poisson manifolds when $k=1$, see
\cite{CDW,McX}.

\begin{proof}
We know from Corollary~\ref{cor:polymu} that multiplicative
poly-symplectic structures on $\mathcal{G}\rightrightarrows M$
correspond to IM poly-symplectic forms $\mu$ on its Lie algebroid
$A\to M$ via
\begin{equation}\label{eq:cond}
i_{u^R}\omega = t^*(\mu(u)),
\end{equation}
and that $\mu$ corresponds to a poly-Poisson structure $(S,P)$ on
$M$, as described in Lemma~\ref{lem:incl}. It remains to verify that
condition \eqref{eq:cond} implies that $t$ is a poly-Poisson map.

Let $\alpha\in \mathrm{Im}(\mu)=S$. Then $\alpha=\mu(u)$ for a
unique $u\in A$, and $P(\alpha)=\rho(u)$. Let $(S_\w, P_\w)$ be the
poly-Poisson structure defined by $\w$, as in \eqref{Eq:S_wP_w}.
Then \eqref{eq:cond} says that $t^* \alpha \in S_\w$ and
$u^R=P_\w(t^*\alpha)$; the fact that on any Lie groupoid we have
$t_*(u^R)= \rho(u)$ implies that $t_* P_\w(t^*\alpha)=P(\alpha)$,
i.e., $t$ is a poly-Poisson map.
\end{proof}


\begin{remark}\label{rem:unique2}
Given a $k$-poly-symplectic groupoid $(\mathcal{G}\rightrightarrows
M,\w)$, the uniqueness of the induced poly-Poisson structure $(S,P)$
on $M$ follows from Remark~\ref{rem:unique}: note that $S$ is
determined by $\omega$, while $P$ is completely defined from the
property that $t$ is a poly-Poisson map.
\end{remark}

We illustrate the correspondence in Theorem~\ref{Thm:integration}
with some simple examples.

\begin{example}\label{Ex:trivialPP}

The $k$-poly-symplectic manifold $\opk \TQ$ of
Example~\ref{Ex:polysymplectic} is a poly-symplectic groupoid over
$Q$, with respect to fibrewise addition; the source and target maps
coincide with the projection $\opk \TQ\to Q$. The corresponding
$k$-poly-Poisson structure on $Q$ is the trivial one, given by
$S:=\opk \TQ$ and $P=0$. Note that Example~\ref{ex:PPtrivial} shows
other poly-Poisson structures on $Q$ for which the projection $\opk
\TQ\to Q$ is a poly-Poisson map, but there is only one with the
bundle $S$ prescribed by Theorem~\ref{Thm:integration}.


\end{example}

\begin{example}
Let $(M,\w)$ be a $k$-poly-symplectic  manifold, that we view as a
poly-Poisson manifold as in Example~\ref{Ex:polysp}. The
non-degeneracy of $\w$ implies that the Lie algebroid $(S_\w,P_\w)$
is isomorphic to $TM$. Hence this poly-Poisson structure is
integrated by the pair groupoid $M\times M\rightrightarrows M$,
equipped with the $k$-poly-symplectic  structure $t^*\omega - s^*\omega \in
\Omega^2(M\times M, \mathbb{R}^k)$, where $s,t$ are the source and target maps
on the pair groupoid, i.e $t(x,y)=x$ and $s(x,y)=y$.
\end{example}

\begin{example}\label{Ex:prod}
Consider Poisson manifolds $(M_j,\pi_j)$, $j=1,\ldots,k$, and equip
$M=M_1\times \ldots\times M_k$ with the product poly-Poisson
structure of Example~\ref{Ex:PPprod}. For each $j$, suppose that
$(\mathcal{G}_j\rightrightarrows M_j, \omega_j)$ is a symplectic
groupoid integrating $(M_j,\pi_j)$. Then the product poly-symplectic
groupoid $\mathcal{G}=\mathcal{G}_1\times \ldots \times
\mathcal{G}_k$ of Prop.~\ref{prop:prod} integrates $M$. Indeed, one
may verify that the bundle $S$ on $M$ described in
Example~\ref{Ex:PPprod} agrees with the one prescribed by
Theorem~\ref{Thm:integration} and, as a consequence of the
construction of poly-Poisson maps as products of Poisson maps in
Example \ref{Ex:PPprod}, the target map on $\gpd\rightrightarrows M$
is a poly-Poisson map.
\end{example}


\begin{example}[Lie-Poisson structures]\label{Ex:int.Lalgb_k}
Let $\G$ be a Lie group and $\LG$ its Lie algebra. As seen in
Example~\ref{Ex:polysymplectic}, $\opk T^*\G$ has a natural
$k$-poly-symplectic structure $\omega$.

The diagonal coadjoint action of $\G$ on $\LG_{(k)}^*$, denoted by
$\Ad_g^*$, endows $\G\times \LG_{(k)}^*$ with a
groupoid structure over $\LG_{(k)}^*$, with source and target maps
given by
\begin{center}\begin{tabular}{cc}
$s(g,\zeta)=\zeta,$&$t(g,\zeta)=\Ad_g^*\zeta$\\
\end{tabular}\end{center}
and multiplication $m((g,\zeta),(h,\eta))=(gh,\eta)$ if $\Ad_h^*\eta=\zeta$.
Using the identification $T^*\G \cong \G \times\LG^*$ (by right translation),
we see that
$$
\opk T^*\G \cong \G\times \LG_{(k)}^*,
$$
so we may consider $\opk T^*\G$ as a Lie groupoid, and its
poly-symplectic structure $\omega$ makes it into a poly-symplectic
groupoid. This structure integrates the direct-sum poly-Poisson
structure on $\LG^*_{(k)}$ described in Example~\ref{Ex:Lalgb_k}.
Indeed, $t$ has the Poisson maps $\G \times \LG^*\to \LG^*$ as its
coordinates, so it is a poly-Poisson map. And one can check that the
bundle $S$ of the direct-sum poly-Poisson structure is the one
induced by the poly-symplectic structure $\omega$ according to
Theorem~\ref{Thm:integration}.


\begin{remark}
More generally: following \cite{IMV} there is a direct-sum
poly-Poisson structure on $A^*\oplus\ldots\oplus A^*$ where $A^*\to
M$ is endowed with the linear Poisson structures (defined on the
dual bundle to the Lie algebroid $A\to M$). Each $A^*$ is integrated
by the symplectic groupoid $T^*\mathcal{G}\rightrightarrows A^*$,
where $\mathcal{G}\rightrightarrows M$ is the groupoid integrating
$A$, and it can be similarly proved that the direct sum
$T^*\mathcal{G}\oplus \ldots\oplus T^*\mathcal{G}$ over $\gpd$ is
the poly-symplectic groupoid integrating $A^*\oplus\ldots\oplus
A^*$.
\end{remark}

\end{example}

\subsection{Poly-symplectic foliation}

It is well known that any Poisson manifold has an underlying
symplectic foliation which uniquely determines the Poisson
structure. More generally, let $(S,P)$ be a $k$-poly-Poisson
structure on $M$. Since $S$ has a Lie algebroid structure, the
distribution $D:= P(S)\subseteq TM$ is integrable, and its leaves
define a singular foliation on $M$. Each leaf $\iota:
\mathcal{O}\hookrightarrow M$ carries an $\mathbb{R}^k$-valued
2-form $\w_\mathcal{O}$ determined by the condition
\begin{equation}\label{Eq:PSFol}
\w_\mathcal{O}^\flat : T\mathcal{O} \to \oplus_{(k)}
T^*\mathcal{O},\;\;\; P(\bar{\eta}) \mapsto \iota^*\bar{\eta}.
\end{equation}
The fact that the 2-form $\omega_\mathcal{O}$ on $\mathcal{O}$ is
well defined follows from (i) in Def.~\ref{Def:PP}, (ii) guarantees
that it is non-degenerate and (iii) that it is closed, see
\cite[Sec.~3]{IMV}. So $(S,P)$ determines a singular foliation on
$M$ with $(k+1)$-poly-symplectic leaves.


A first remark on the poly-symplectic foliation of a
$k$-poly-Poisson structure is that, in contrast with the case $k=1$,
different $k$-poly-Poisson structures may correspond to the same
poly-symplectic foliation, as shown in the next example.

\begin{example}\label{Ex:PSfoliation is not unique}
Let $\w_t$ be a smooth family of $k$-poly-symplectic forms on
$M$ parametrized by $t\in \mathbb{R}$ and define the following
vector subbundles of $\opk T^*(M\times \mathbb{R})$:
\begin{center}\begin{tabular}{l}
$S_1\big{|}_{(m,t)}:=\{(i_X\w_t,r_1\oplus \cdots \oplus r_k)|X\in T_mM, r_j\in T_t^*\mathbb{R}\},$\\
$S_2\big{|}_{(m,t)}:=\{(i_X\w_t,r\oplus \cdots \oplus r)|X\in T_mM, r\in T_t^*\mathbb{R}\},$\\
$S_3\big{|}_{(m,t)}:=\{(i_X\w_t,r\oplus 0 \cdots \oplus 0)|X\in T_mM,
r\in T_t^*\mathbb{R}\};$
\end{tabular}\end{center}
on each $S_j$ we define $P_j(i_X\w_t,\bar{\gamma})=X.$ Observe that
each $(S_j,P_j)$ is a poly-Poisson
structure on $M\times \mathbb{R}$ 
but these three $k$-poly-Poisson structures have the same
poly-symplectic foliation on $M\times \mathbb{R}$.
Same conclusion holds for the weak-poly-Poisson structure given by
$$S_0\big{|}_{(m,t)}:=\{(i_X\w_t,0)|X\in T_mM\} \mbox{\ \ and\ \ }P_0(i_X\w_t,0)=X$$
where the poly-symplectic foliation is described on Theorem 3.4 on \cite{IMV}.
\end{example}

We now discuss the possibility of defining a poly-Poisson structure
from a poly-symplectic foliation. Given a subspace $D_m\subseteq
T_mM$, for $m\in M$, equipped with a $(k+1)$-poly-symplectic form
$\omega_m$, we consider the subspace $S_m\subseteq \opk T^*M$ given
by
\begin{equation}\label{eq:Sm}
S_m:=\{\bar{\eta}\in \opk T_m^*M\,|\,\exists X\in D_m,\;
\bar{\eta}\big{|}_{\opk D_m}=i_X\w_m \},
\end{equation}
which has dimension $k(n-p) + p$, where $p$ is the dimension of
$D_m$. One may verify that $S_m^\circ = \{0\}$ and there is a well-defined map $P_m: S_m\to D_m\subseteq T_mM$,
\begin{equation}\label{eq:Pm}
P_m(\bar{\eta}) = X, \; \mbox{ where }\; \bar{\eta}\big{|}_{\opk
D_m}=i_X\w_m.
\end{equation}
Given now a {\it regular} poly-symplectic foliation on $M$, letting
$D$ be its tangent distribution, we use the previous pointwise
construction to see that \eqref{eq:Sm} defines a subbundle
$S\subseteq \opk T^*M$, satisfying $S^\circ=\{0\}$, and equipped
with a bundle map $P: S\to TM$. Moreover, using the fact that the
$\mathbb{R}^k$-valued form defined on each leaf is closed, it
follows that $(S,P)$ satisfies (iii) in Def.~\ref{Def:PP}, so it is
a poly-Poisson structure. In conclusion we have the following
proposition (see \cite[Sec.~3]{IMV}),
\begin{proposition}
If $(D,\w)$ is a regular $k$-poly-symplectic foliation on $M$ then
$(S,P)$, defined pointwise by \eqref{eq:Sm} and \eqref{eq:Pm}, is a
$k$-poly-Poisson on $M$.
\end{proposition}
In particular, if the regular $k$-poly-symplectic foliation on $M$ comes from
a weak-poly-Poisson structure as in \cite[Theorem 3.4]{IMV}, then the poly-Poisson
structure on the proposition is an ``extension'' of the weak-poly-Poisson structure 
In order to illustrate last claim and the poly-Poisson structure from \eqref{eq:Sm} and
\eqref{eq:Pm} we apply the proposition to the regular poly-symplectic foliation given in
Example \ref{Ex:PSfoliation is not unique}, which is the same for each poly-Poisson strucutres
$(S_j,P_j)$ for $j=1,2,3$ and for the weak-poly-Poisson $(S_0,P_0)$, and
get the ``maximal'' poly-Poisson structure $(S_1,P_1)$.

\subsection{Relation with AV-Dirac structures}

It is well known that Poisson structures on $M$ can be understood as
special types of Dirac structures in the Courant algebroid $TM\oplus
T^*M$ \cite{Co}. As we now see, this picture can be generalized to
poly-Poisson structures. We consider the bundle
$\mathbb{A}:=TM\oplus (\opk \TM)$, equipped with the
($\mathbb{R}^k$-valued) fibrewise inner product
$$
\langle X\oplus \bar{\eta}, Y\oplus \bar{\gamma}\rangle
:=i_X\bar{\gamma}+i_Y\bar{\eta},
$$
and bracket on sections of $\mathbb{A}$ given by
$$
\ca X\oplus \bar{\eta}, Y\oplus \bar{\gamma} \cc
:= [X,Y]\oplus \Lie_X \bar{\gamma} - i_Yd\bar{\eta}.
$$
For $k=1$, this is the standard Courant algebroid $TM\oplus T^*M$.
In general, this is a very particular case of the {\it AV-Courant
algebroids} introduced in \cite[Sec.~2]{LiB} (with respect to the
Lie algebroid $A=TM$ and representation on $V=M\times \mathbb{R}^k
\to M$ given by the Lie derivative $\Lie_X(f_1,\dots
,f_k)=(\Lie_Xf_1,\dots ,\Lie_Xf_k)$ on
$C^{\infty}(M,\mathbb{R}^k)$).

Following \cite{LiB}, one may consider {\it AV-Dirac structures} on
any AV-Courant algebroid: these are subbundles $L\subseteq
\mathbb{A}$ which are {\it lagrangian}, i.e.,
\begin{equation}\label{eq:lag}
L=L^\perp,
\end{equation}
with respect to the fibrewise inner product, and which are
involutive with respect to the bracket  $\ca \cdot,\cdot \cc$ on
$\Gamma(\mathbb{A})$. Recall that $L$ is called {\it isotropic} if
$L\subset L^\perp$.

\begin{example}
Any $k$-poly-symplectic structure $\w$ on $M$ may be seen as an
AV-Dirac structure in $\mathbb{A}:=TM\oplus (\opk \TM)$ via
$$
L:=\mathrm{graph}(\w)=\{X\oplus i_X\w\,|\,X\in TM\}.
$$
Note that this $L$ satisfies the additional condition
\begin{equation}\label{eq:trans1}
L\cap (\opk \TM)=\{0\}.
\end{equation}
In fact, poly-symplectic structures on $M$ are in one-to-one correspondence with
AV-Dirac structures which project isomorphically over $TM$ and satisfy
 $L\cap TM=\{0\}$ and \eqref{eq:trans1}.
\end{example}

Our goal now is to define, in the same way, a subbundle $L$ from a
poly-Poisson structure $(S,P)$, i.e. consider
\[
L=\{P(\bar{\eta})\oplus \bar{\eta}|\bar{\eta}\in S\}.
\]
Note that $L$ is isotropic as a consequence of (i) in Definition
\ref{Def:PP}. But, as we now see, the lagrangian condition generally
fails.

\begin{example} Let $\LG$ be a Lie algebra and consider the poly-Poisson
structure on $\LG_{(2)}^*$ as in Example \ref{Ex:Lalgb_k}. Observe
that  $L$ over the point $\zeta=(0,0)\in \LG_{(2)}^*$  can be
written as
\[
L_\zeta=\{(0,0)\oplus ((u,0),(0,u))|u\in \LG\}.
\]
But for any $v_1,v_2,w_1,w_2\in \LG$ we have
$(0,0)\oplus((v_1,v_2),(w_1,w_2))\in L^{\perp}$, hence $L$ is
properly contained in $L^\perp$.
\end{example}

Therefore, in general, poly-Poisson structures are not AV-Dirac
structures. In order to include poly-Poisson structures in the
formalism of AV-Courant algebroids, one then needs to relax the
lagrangian condition \eqref{eq:lag}.

Let us consider subbundles $L\subseteq TM\oplus (\opk \TM)$
satisfying
\begin{equation}\label{eq:wlag}
L = L^\perp \cap (L+TM).
\end{equation}
Note that \eqref{eq:lag} implies that \eqref{eq:wlag} holds, but the
converse is not true.


The following results characterize $k$-poly-Poisson structures as
subbundles of $\mathbb{A} = TM\oplus (\opk \TM)$:

\begin{proposition}\label{eq:charactL}
There is a one-to-one correspondence among the following:
\begin{itemize}
\item[(a)] $k$-poly-Poisson structures $(S,P)$ on $M$,
\item[(b)] Involutive, isotropic subbundles $L\subset \mathbb{A}$ satisfying $L^\perp\cap
TM=\{0\}$,
\item[(c)] Involutive subbundles $L\subset \mathbb{A}$ satisfying
\eqref{eq:wlag} and $L\cap TM=\{0\}$.
\end{itemize}
\end{proposition}

\begin{proof}

Given a $k$-poly-Poisson structure $(S,P)$, we define the subbundle
$L\subset \mathbb{A}$ by
\begin{equation}\label{eq:LP}
L=\{P(\bar{\eta})\oplus \bar{\eta}\,|\,\bar{\eta}\in S\}.
\end{equation}
This bundle is isotropic by condition (i) in Def.~\ref{Def:PP},
condition (ii) amounts to $L^\perp\cap TM=\{0\}$ while (iii) is
equivalent to the involutivity of $L$. Conversely, given $L$ as in
(b), the image of the natural projection $L\to \opk T^*M$ defines a
vector bundle $S$ and a bundle map $P: S\to TM$ by
$$
P(\bar{\eta})=X \;\mbox{ if and only if }\; X\oplus \bar{\eta}\in L,
$$
in such a way that $(S,P)$ is a $k$-poly-Poisson structure. This
gives the correspondence between (a) and (b).

For a $k$-poly-Poisson structure $(S,P)$ and $L$ as in
\eqref{eq:LP}, one may directly verify that (i) in Def.~\ref{Def:PP}
implies that \eqref{eq:wlag} holds, while (ii) implies that $L\cap
TM=\{0\}$, so $L$ satisfies the properties in (c). It remains to
check that given an $L$ as in (c), then it satisfies the properties
described in (b). Note that \eqref{eq:wlag} implies that $L$ is
isotropic and that $L\cap TM=L^\perp\cap TM$, so that $L^\perp\cap
TM=\{0\}$.

\end{proof}
\begin{remark} For $k=1$, the objects in (b) and (c) are just usual Dirac
structures on $M$, satisfying the additional condition $L\cap
TM=\{0\}$ (conditions \eqref{eq:lag} and \eqref{eq:wlag} turn out to
be equivalent for $k=1$), while the objects in (a) are usual Poisson
structures. So for $k=1$ Prop.~\ref{eq:charactL} boils down to the
known characterization of Poisson structures as particular types of
Dirac structures.
\end{remark}


\section{Symmetries and reduction}\label{sec:sym}

We now discuss poly-Poisson structures and poly-symplectic groupoids
in the presence of symmetries, with the aim of using reduction as a
tool for integration of poly-Poisson manifolds, along the lines of
\cite{MiWe,FOR}.


\subsection{Poly-Poisson actions}

An action $\varphi$ of a Lie group $\G$ on a $k$-poly-Poisson
manifold $(M,S,P)$ is a \textit{poly-Poisson action} if for each
$g\in \G$ the diffeomorphism $\varphi_g:M\to M$ is a poly-Poisson
morphism (Def. \ref{Def:mor}). In the case of $k$-poly-symplectic
manifold $(M,\omega)$, this means that $\varphi_g^*\omega=\omega$,
see Example~\ref{Ex:polysp}.



Let us consider a poly-Poisson action $\varphi$ of a Lie group $\G$
on $(M,S,P)$, and let us assume henceforth that this action is free
and proper, so that we have a principal $\G$-bundle:
\begin{equation}\label{eq:proj}
\Pi:M\to M/\G.
\end{equation}
Let $V\subseteq TM$ denote the vertical bundle defined by this
action.

It is well-known that, when $k=1$,  i.e.,  $M$ is an ordinary
Poisson manifold, $M/\G$ inherits a Poisson structure for which
$\Pi$ is a Poisson map. For poly-Poisson manifolds, we will need
additional conditions. We call the action $\varphi$ is
\textit{reducible} if
\begin{equation}\label{Def:p.reductible}
\begin{cases}
(a)\ S\cap \oplus_k \Ann(V) \mbox{\ \ has constant rank\ }, \\
(b)\ (S\cap \oplus_k \Ann(V))^\circ\subset V.
\end{cases}
\end{equation}
The projection map \eqref{eq:proj} induces a map $d\Pi_{(k)}: \opk
TM\to \Pi^*(\opk T(M/\G))$, and its transpose is an injective bundle
map $\Pi^*(\opk T^*(M/\G))\to \opk \TM$, whose image is the
subbundle $\opk \Ann(V) \subseteq \opk \TM$. So we have an induced
isomorphism
\begin{equation}\label{eq:transpose}
\Pi^*(\opk T^*(M/\G))\stackrel{\sim}{\to} \opk \Ann(V).
\end{equation}

The next result is analogous to \cite[Thm.~4.1]{IMV} (but stated for
our stronger notion of poly-Poisson structure).

\begin{theorem}\label{Thm:PPred}
Let us consider a poly-Poisson $\G$-action on a $k$-poly-Poisson
manifold $(M,S,P)$ which is free and proper, and reducible. Then
$M/\G$ inherits a $k$-poly-Poisson  structure $(S_{red},P_{red})$, where the
subbundle $S_{red} \subseteq \opk T^*(M/\G)$ corresponds to $S\cap \opk
\Ann(V)$ via \eqref{eq:transpose}, and $P_{red}$ is unique so that the
quotient map \eqref{eq:proj} is a $k$-poly-Poisson morphism.

\end{theorem}

\begin{proof}
The first condition in \eqref{Def:p.reductible} guarantees that
$S_{red}\subseteq \opk T^*(M/\G)$, defined by the condition that
$\Pi^*S_{red}$ is isomorphic to $S\cap \opk \Ann(V)$ under
\eqref{eq:transpose}, is a vector subbundle. Note that we have a
natural map $\Pi^*(S_{red})\to \Pi^*(T(M/\G))$ given by the composition
\begin{equation}\label{eq:compos}
\Pi^*(S_{red})\stackrel{d\Pi^*_{(k)}}{\longrightarrow} S\cap \opk
\Ann(V) \stackrel{P}{\longrightarrow}
TM\stackrel{d\Pi}{\longrightarrow} \Pi^*(T(M/\G)),
\end{equation}
and this defines a bundle map
\begin{equation}\label{eq:Pr}
P_{red}: S_{red}\to T(M/\G)
\end{equation}
as a consequence of the $\G$-invariance of $(S,P)$.

To check that $(S_{red},P_{red})$ defines a $k$-poly-Poisson structure on
$M/\G$, one must verify that it satisfies conditions (i), (ii), and
(iii) in Def.~\ref{Def:PP}. Condition (i) follows directly from the
definition of $(S_{red},P_{red})$ and the fact that this condition is
satisfied by $(S,P)$. It is also routine to check that condition
(iii) holds for $(S_{red},P_{red})$, given that it holds for $(S,P)$.

As for condition (ii), it is a consequence of property (b) in
\eqref{Def:p.reductible}. Indeed, by the way $S_{red}$ is defined, the
fact that $\bar{X} \in S_{red}^\circ$ implies that $\bar{X}=d\Pi(X)$,
for $X\in (S\cap \opk \Ann(V))^\circ$. But then (b) in
\eqref{Def:p.reductible} implies that $\bar{X}=d\Pi(X)=0$.

It is also clear from the definition of $P_{red}$ that $\Pi$ is a
poly-Poisson map.

\end{proof}

We mention two concrete examples, discussed in \cite{IMV}.

\begin{example}\label{ex:ppred}
\
\begin{itemize}
\item[(a)]
Let $Q$ be a manifold equipped with a free and proper $\G$-action,
and let $(M=\oplus_{(k)}T^*Q, \w)$  be the poly-symplectic manifold
of Example~\ref{Ex:polysymplectic}. We keep the notation $\pr_j:
M\to T^*Q$ for the natural projection onto the $j$th-factor. The
cotangent lift of the $\G$-action on $Q$ defines an action on
$T^*Q$, which induces a $\G$-action on $M$ which preserves the
poly-symplectic structure (i.e., it is a poly-Poisson action), and
there is a natural identification
$$
M/\G \cong \oplus_{(k)} (T^*Q/\G).
$$
We observe here that both conditions in \eqref{Def:p.reductible}
hold, i.e., the $\G$-action on $M$ is reducible. To verify this
fact, let $V\subseteq TM$ be the vertical bundle of the $\G$-action
on $M$, so that $V_j=d\pr_j(V) \subseteq T(T^*Q)$ is the vertical
bundle of the $\G$-action on the $j$th-factor $T^*Q$. Note that the
natural projection $T^*Q\to Q$ induces a projection of
$V_j^{\w_{can}}$ onto $TQ$, and one then sees that
$$
V_1^{\w_{can}} \times_{TQ}\ldots \times_{TQ} V_k^{\w_{can}}\subseteq
T(T^*Q) \times_{TQ}\ldots \times_{TQ} T(T^*Q)= TM
$$
is a vector subbundle, that we denote by $W$. One can now check that
\begin{equation}\label{eq:Scovel}
S_\w \cap \opk \Ann(V)=\{i_X\w\,|\, X\in W\},
\end{equation}
from where one concludes that condition (a) of
\eqref{Def:p.reductible} holds. From \eqref{eq:Scovel}, one directly
sees that
\begin{align*}
(S_\w\cap \opk \Ann(V))^\circ &= (V_1^{\w_{can}})^{\w_{can}}
\times_{TQ}\ldots \times_{TQ} (V_k^{\w_{can}})^{\w_{can}}\\
& = V_1 \times_{TQ}\ldots \times_{TQ} V_k = V,
\end{align*}
showing that (b) of \eqref{Def:p.reductible} also holds. So the
action is reducible. As shown in \cite[Ex.~4.3]{IMV}, the reduced
poly-Poisson structure on $\oplus_{(k)} (T^*Q/\G)$ is the one
defined by direct-sum of the natural linear Poisson structure on
$T^*Q/\G$ (dual to the Atiyah algebroid $TQ/\G$ of the principal
bundle $Q\to Q/\G$).

\item[(b)] In the particular case of $Q=\G$ with the action by left
multiplication, as shown in \cite[Ex.~4.2]{IMV}, the poly-Poisson
reduction of $\opk T^*\G$ with respect to the lifted $\G$-action is
identified with $\LG^*_{(k)}$ of Example~\ref{Ex:Lalgb_k}.
\end{itemize}
\end{example}

\subsection{Hamiltonian actions on poly-symplectic manifolds}

We now consider poly-Poisson actions on poly-symplectic manifolds in
the presence of moment maps.

Let $(M,\w)$ be a $k$-poly-symplectic manifold equipped with a
poly-Poisson action of $\G$, denoted by $\varphi$. Consider the
diagonal coadjoint action of $\G$ on the space $\LG^*_{(k)}$. This action is called {\em
hamiltonian} \cite{Gu,MR-RSV} if there is a {\em moment map}, i.e., a map $J:M\to
\LG^*_{(k)}$ that satisfies
\begin{equation}\label{Def:HamPPaction}
\begin{tabular}{rlclr}
(i)&$J\circ \varphi_g=\Ad_g^*\circ J$& and &(ii)&$i_{u_M}\w=d\langle
J,u\rangle$.\\
\end{tabular}
\end{equation}
for all $u\in \LG$. Here $u_M\in \mathfrak{X}(M)$ denotes the
infinitesimal generator corresponding to $u\in \LG$.

\begin{example}\label{Ex:MMexct.mfd}
Let $(M,\w)$ be a poly-symplectic manifold such that $\w=-d\theta $,
and assume that $\G$ acts on $M$ preserving the 1-form $\theta$.
Then the maps $J_1,\dots J_k:M\to \LG^*$ defined by $\langle
J_l,u\rangle=\theta_l(u_M)$, $u\in \LG$, define a moment map for the
action.

A particular case of this example is when $M=\opk T^*Q$ (as in
Example~\ref{Ex:polysymplectic}) and the action of $\G$ on $M$ is
the lift of an action on $Q$, see Example~\ref{ex:ppred}(a). Here
the moment map is $\langle J(\bar{\eta}),u\rangle =(\langle
\eta_j,u_Q\rangle)_{j=1,\dots,k}$.
\end{example}

The following observation generalizes a well-known fact in Poisson
geometry. Consider $\LG_{(k)}^*$ with the poly-Poisson structure of
Example \ref{Ex:Lalgb_k}.

\begin{proposition}\label{Prop:MMisPPmorphism}
The moment map $J:M\to \LG_{(k)}^*$ of a hamiltonian action of $\G$
on $(M,\w)$ is a poly-Poisson morphism.
\end{proposition}

\begin{proof}
Denote the poly-Poisson structure on $\LG_{(k)}^*$ by $(S,P)$, as in
Example \ref{Ex:Lalgb_k}.  Consider $(u_1,\dots ,u_k)\in S|_\zeta$,
and $Y\in T_xM$ with $J(x)=\zeta$. By condition (ii) in
\eqref{Def:HamPPaction} we have
\begin{align*}
(J^*(u_1,\dots ,u_k))(Y)&=(u_1,\dots ,u_k)(dJ(Y))=(\langle dJ(Y),u_j\rangle)_{j=1,\dots ,k}\\
&=(\langle dJ_j(Y),u\rangle)_{j=1,\dots ,k}=(i_{u_M}\w)\big{|}_x(Y),
\end{align*}
hence $(J^*(u_1,\dots ,u_k))=i_{u_M}\w\big{|}_x\in
\mathrm{Im}(\w^\flat)\big{|}_x$.

Recall the bundle maps of the poly-Poisson structures on $M$ and
$\LG_{(k)}^*$:
$$
P_\w(J^*(u_1,\dots
,u_k))\big{|}_x=P_\w(i_{u_M}\w\big{|}_x)=u_M\big{|}_x,
$$
$$
P|_\zeta(u_1,\dots ,u_k)=(\mathrm{ad}_u^*\zeta_j)_{j=1,\dots
,k}=(u_{\LG^*}(\zeta_j))_{j=1,\dots ,k}.
$$
>From condition (i) in \eqref{Def:HamPPaction} we can derive that
$dJ_j(u_M(x))=u_{\LG^*}(\zeta_j)$, therefore on points $\zeta=J(x)$
we obtain
$$
dJ(P_\w|_x(J^*(u_1,\dots ,u_k))=(dJ_j(u_M(x)))=(u_{\LG^*}(\zeta_j))
=P|_{\zeta}(u_1,\dots ,u_k).
$$
\end{proof}

Let us consider a Hamiltonian $\G$-action on a $k$-poly-symplectic
manifold $(M,\w)$, with moment map $J:M\to \LG^*_{(k)}$. Let
$\zeta\in \LG_{(k)}^*$ be a \textit{clean value} for $J$, i.e.,
\begin{equation}\label{Def:cleanvalue}
\begin{cases}
 J^{-1}(\zeta) \text{ \ is a submanifold of\ }M,\\
 \ker(d_xJ)=T_xJ^{-1}(\zeta), \text{ \ for all \ } x\in  J^{-1}(0).
\end{cases}
\end{equation}
The submanifold $J^{-1}(\zeta)$ is invariant by the action of
$\G_\zeta$, the isotropy group of $\zeta$ with respect to the
diagonal coadjoint action. We assume that the $\G_\zeta$-action on
$J^{-1}(\zeta)$ is free and proper, so we can consider the reduced
manifold
$$
M_\zeta:=J^{-1}(\zeta)/\G_{\zeta}.
$$
We let $\Pi_\zeta:J^{-1}(\zeta)\to M_\zeta$ be the natural
projection map, and $i_\zeta:J^{-1}(\zeta) \to M$ the inclusion. We
denote by $V_\zeta \subseteq T J^{-1}(\zeta)$ the vertical bundle
with respect to the $\G_\zeta$-action. It follows from (i) in
\eqref{Def:HamPPaction} that $V_\zeta = V\cap TJ^{-1}(\zeta)$, while
(ii) implies that
\begin{equation}\label{eq:Vzeta}
V_\zeta \subseteq \ker(i_\zeta^*\omega).
\end{equation}
This last condition, together with the $\G_\zeta$-invariance of
$i_\zeta^*\omega$, implies that $i_\zeta^*\omega$ is basic, i.e.,
there exists a (unique) closed form $\w_{red}\in \Omega^2(M_\zeta,
\mathbb{R}^ k)$ so that
\begin{equation}\label{eq:wr}
\Pi_\zeta^*\w_{red}=i_\zeta^*\w.
\end{equation}
In general, however, the form $\omega_{red}$ fails to be
poly-symplectic, as it may be degenerate; indeed, it is
nondegenerate if and only if we have an equality in
\eqref{eq:Vzeta}.

Note that (ii) in \eqref{Def:HamPPaction} says that
$$
TJ^{-1}(\zeta) = \ker(dJ) = V^\omega,
$$
where, for a subbundle $W\subseteq TM$, we use the notation
$W^\omega=\{Y\in TM ,|\, \w(X,Y)=0 \; \forall X\in W\}$. Writing
$S=\mathrm{Im}(\w^\flat)$, one may also check that

$$
(\ker(dJ))^\omega = (V^\omega)^\omega = (S\cap \oplus_k
\Ann(V))^\circ,
$$
from where we conclude that
$$
\ker(i_\zeta^*\omega) = (TJ^{-1}(\zeta))^\omega\cap TJ^{-1}(\zeta) =
(S\cap \oplus_k \Ann(V))^\circ \cap T J^{-1}(\zeta).
$$
Comparing with \eqref{eq:Vzeta}, we conclude the following:

\begin{proposition}\label{prop:MW}
The reduced form $\omega_{red} \in \Omega^2(M_\zeta, \mathbb{R}^ k)$
defined by \eqref{eq:wr} is poly-symplectic if and only if
\begin{equation}\label{eq:redpoly}
(S\cap \oplus_k \Ann(V))^\circ\cap TJ^{-1}(\zeta) \subseteq V_\zeta
= V\cap TJ^{-1}(\zeta).
\end{equation}
\end{proposition}
A similar, but not equivalent, result of the previous condition was stated on
\cite[Lemma 3.16]{MR-RSV}. 
\begin{example}
Consider symplectic manifolds $(M_j,\w_j)_{j=1,\dots,k}$ each of them carrying a
Hamiltonian action of a Lie group $\G_j$ with respective moment map $J_j:M_j\to \LG_j^*$.
On the product $k$-poly-symplectic manifold $(M,\w)$ (see Section \ref{subsec:polysymp})
there is a poly-symplectic hamiltonian action given by the product action of
$\G:=\prod _{j=1}^k \G_j$ on $M$ and the moment map
$J:M\to \opk (\prod_{j=1}^k \LG_j^*)$, $J(m)=\oplus_{j=1}^k (0,\dots,0,J_j(m_j),0,\dots,0)$.

Let $\zeta=\oplus_{j=1}^k(0,\dots,0,\zeta_j,0,\dots,0)\in \opk
(\prod_{j=1}^k \LG_j^*)$ where $\zeta_j\in \LG_j^*$ is a clean value
for $J_j$. Then $J^{-1}(\zeta)=\prod_{j=1}^k J_j^{-1}(\zeta_j)$ and,
assuming that each $\G_{\zeta_j}$ acts freely and properly on
$J_j^{-1}(\zeta_j)$, then
$$M_\zeta:=J^{-1}(\zeta)/\G_\zeta=\prod_{j=1}^k J_j^{-1}(\zeta_j)/\G_{\zeta_j}=\prod_{j=1}^k M_{j,\zeta_j},$$
and the reduced $\mathbb{R}^k$-valued 2-form on $M_\zeta$ is the product $k$-poly-symplectic
form defined by the reduced symplectic forms on $M_{j,\zeta_j}$.

\end{example}

The moment-map reduction of Prop.~\ref{prop:MW} can now be compared
with the quotient of poly-Poisson structures in
Theorem~\ref{Thm:PPred}.

Assuming that the $\G$-action on $M$ is free and proper, and that
$\zeta$ is a clean value of a moment map $J: M\to \LG_{(k)}^*$, it
follows that the $\G_\zeta$-action on $J^{-1}(\zeta)$ is also free
and proper, and we have the following diagram of submersions and
natural inclusions:
\begin{equation}\label{Eq:Diagram projection}
\begin{diagram}
\node{J^{-1}(\zeta)} \arrow{e,t}{i_\zeta} \arrow{s,l}{\Pi_\zeta} \node{M} \arrow{s,r}{\Pi}\\
\node{M_\zeta} \arrow{e,t}{}\node{M/\G}.
\end{diagram}
\end{equation}


\begin{proposition}\label{Prp:Reduction}
Let $(M,\w)$ be a poly-symplectic manifold equipped with a
hamiltonian $\G$-action with moment map $J:M\to \LG^*_{(k)}$. Assume
that the $\G$-action on $M$ is free, proper and reducible
\eqref{Def:p.reductible}. If $\zeta\in \LG_{(k)}^*$ is a clean value
for the moment map, then:
\begin{itemize}
\item[(a)] The reduced manifold $M_\zeta=J^{-1}(\zeta)/\G_\zeta$ carries a natural
poly-symplectic form defined by equation \eqref{eq:wr};
\item[(b)] The poly-symplectic manifold $M_\zeta$ sits in $M/\G$ as a union of
poly-symplectic leaves of the reduced poly-Poisson manifold on
$M/\G$ (given by Thm.~\ref{Thm:PPred}).
\end{itemize}
\end{proposition}

\begin{proof}
Note that \eqref{Def:p.reductible}(b) directly implies
\eqref{eq:redpoly}, so the reduced form $\omega_{red}$ on $M_\zeta$ is
indeed poly-symplectic, proving part (a).

By the moment-map condition \ref{Def:HamPPaction}(ii), $X\in
\ker(dJ)$ if and only if $(i_X\w)(u_M)=0$ for all $u\in \LG$,
therefore
\begin{align}\label{Eq:TJ^-1(z)}
TJ^{-1}(\zeta)&=\{X\in TM| i_X\w\in \oplus_k \Ann(V) \}\\ \notag
&=P_\omega(S_\omega\cap \oplus_k \Ann(V))= P_\omega(d\Pi_{(k)}^*S_{red}).
\end{align}
It follows from \eqref{Eq:Diagram projection} and the construction of
the reduced poly-Poisson structure, see \eqref{eq:compos} and
\eqref{eq:Pr}, that
$$
TM_\zeta=d\Pi_\zeta(TJ^{-1}(\zeta))=d\Pi(P_\omega(d\Pi_{(k)}^*S_{red}))=P_{red}(S_{red}).
$$
Hence $M_\zeta$ is a union of poly-symplectic leaves in $M/\G$. It
remains to check that the poly-symplectic structures (the one coming
from reduction and the one induced from the poly-Poisson structure
on $M/\G$) agree.

Consider $\bar{X}=d\Pi_\zeta(X)$, $\bar{Y}=d\Pi_\zeta(Y) \in
TM_\zeta$, with $X$, $Y$ tangent to $J^{-1}(\zeta)$, and let us
compute the two 2-forms on them. For the leafwise poly-symplectic
form $\omega_L$, we have (see \eqref{Eq:PSFol})
$$
\w_L(\bar{X},\bar{Y}) = \bar{\eta}_{r}(\bar{Y})
=(\Pi_\zeta^*\w_L)(X,Y),
$$
where $\bar{\eta}_{r}$ is such that $\bar{X}=P_{red}(\bar{\eta}_{r})$.
Letting $\bar{\eta} = d\Pi_{(k)}^*(\bar{\eta}_{r}) \in S_\w\cap
\oplus_k \Ann(V),$ then
$$
(\Pi_\zeta^*\w_L)(X,Y)=\bar{\eta}_{r}(d\Pi_\zeta(Y))=\bar{\eta}(Y).
$$

Note that there exists a unique $X_0\in TM$ such that
$\bar{\eta}=i_{X_0}\w\in \oplus_k\Ann(V)$. By \eqref{Eq:TJ^-1(z)},
we know that $X_0\in TJ^{-1}(\zeta)$. Furthermore,
$$
d\Pi_\zeta (X_0)= d\Pi(P_\w(\bar{\eta})) = d\Pi(P_\w
(d\Pi_{(k)}^*(\bar{\eta}_{r})))=P_{red}(\bar{\eta}_{r}) = \bar{X},
$$
so $d\Pi_\zeta (X_0) = d\Pi_\zeta(X)$. Recalling that
$\Pi_\zeta^*\w_{red}=i_\zeta^*\w$, we see that
$$
(\Pi_\zeta^*\w_{red})(X,Y) = (\Pi_\zeta^*\w_{red})(X_0,Y) =(i_{X_0}\w)(Y)
=\bar{\eta}(Y)=(\Pi_\zeta^*\w_L)(X,Y),
$$
showing that $\w_{red}=\w_L$ on $M_\zeta$.
\end{proof}

\begin{example}\label{ex:redleaf}
\
\begin{itemize}
\item[(a)] Let us consider a $\G$-action on $Q$ and its lift to $M=\opk
T^*Q$ as in Example~\ref{ex:ppred}(a). The action on $M$ is
hamiltonian, and using the explicit formula for the moment map in
Example~\ref{Ex:MMexct.mfd} one sees that its poly-symplectic
reduction at $\zeta=0$ is $\oplus_{(k)} T^*(Q/\G)$, with the
poly-symplectic form of Example~\ref{Ex:polysymplectic}.
Proposition~\ref{Prp:Reduction}(b) realizes $\oplus_{(k)} T^*(Q/\G)$
as a poly-symplectic leaf of $M/\G$.

\item[(b)] Following Example~\ref{ex:ppred}(b), in the particular case $Q=\G$
Proposition~\ref{Prp:Reduction}(b) implies that the poly-symplectic
reduction of the lifted action on $\opk T^*\G$ at level $\zeta$ (see
\cite[Sec.~3.3.2]{MR-RSV}) is identified with the poly-symplectic
leaf of $\LG^*_{(k)}$ through $\zeta$, which is the orbit of $\zeta$
under the diagonal coadjoint action of $\G$ on $\LG^*_{(k)}$ (c.f.
Example~\ref{Ex:int.Lalgb_k}) equipped with a poly-symplectic
generalization of the usual KKS symplectic form on coadjoint orbits,
see \cite[Example~2.9]{IMV} and \cite[App.~A.3]{MR-RSV}.
\end{itemize}
\end{example}

\subsection{Reduction and integration}

In this section, we show (along the lines of \cite{BC2,FOR})
 how passing from poly-Poisson manifolds to poly-symplectic groupoids has the effect
of turning poly-Poisson actions into hamiltonian actions, and how
poly-symplectic reduction can be used in the construction of
poly-symplectic groupoids associated with poly-Poisson quotients.

In the remainder of this section, we will consider the following
set-up:

\begin{enumerate}[1.]
\item A $k$-poly-Poisson manifold $(M,S,P)$, so that its underlying
Lie algebroid is integrable, and $(\mathcal{G}\rightrightarrows
M,\w)$ the source-simply connected $k$-poly-symplectic groupoid
integrating it.

\item A poly-Poisson action $\varphi$ of the Lie group $\G$ on
$(M,S,P)$.
\end{enumerate}

Since $\varphi$ preserves the poly-Poisson structure on $M$, the
cotangent lift of $\varphi$ induces an action $ \hat{\varphi}:
\G\times S\to S $ by Lie-algebroid automorphisms, which can be
integrated to a poly-symplectic $\G$-action on $\mathcal{G}$,
denoted by
$$
\Phi: \G \times \mathcal{G} \to \mathcal{G}.
$$
We will now see that this action on $\mathcal{G}$ admits a natural
moment map (as in \eqref{Def:HamPPaction}), so it is hamiltonian.

Let us start by recalling that any action on $M$ induces a
Hamiltonian $\G$-action on the symplectic manifold $T^*M$ with
moment map $J_{can}:T^*M\to \LG^*$ given by
$$\langle
J_{can}(\alpha),u\rangle=\langle \alpha,u_M\rangle
$$
for all $\alpha\in T^*M$ and $u$ in the Lie algebra $\LG$ of $\G$.
We have an induced map $\opk T^*M \to \LG^*_{(k)}$, that we restrict
to $S$ to define
\begin{equation}\label{Fu:Jsj}
J^s:S \to \LG_{(k)} ^*.
\end{equation}
It is clear from the $\G$-equivariance of  $J_{can}$ that $J^s$ is
$\G$-equivariant (with respect to the diagonal coadjoint action on
$\LG_{(k)} ^*$).

The same proof as in \cite[Lemma~3.1]{BC2} shows that, viewing
$\LG_{(k)} ^*$ as a trivial Lie algebra, $J^s$ is a Lie-algebroid
morphism. According to our sign conventions, it is more convenient
to consider $-J^s$, which is also a Lie-algebroid morphism, and
integrate it to a Lie-groupoid morphism
\begin{equation}\label{eq:momG}
J: \mathcal{G}\to \LG_{(k)} ^*.
\end{equation}
Just as in \cite[Prop.~3.2]{BC2} one can verify that $J$ is
$\G$-equivariant and satisfies:
$$
i_{u_\mathcal{G}}\w = d\langle J,u \rangle,
$$
for all $u\in \LG$, where $u_\mathcal{G}$ is the infinitesimal
generator for the action on $\mathcal{G}$. In other words, $J$ is a
moment map for the action $\Phi$ on $\mathcal{G}$. The next result
summarizes the discussion:

\begin{proposition}
\label{Prp:MMgpd} The $\G$-action $\Phi$ on the poly-symplectic
groupoid $(\mathcal{G},\w)$ is Hamiltonian with moment map
\eqref{eq:momG}.
\end{proposition}

We now discuss the connection between integration and reduction. We
assume from now on that the $\G$-action $\varphi$ on $M$ is free,
proper and reducible \eqref{Def:p.reductible}. Then the action
$\Phi$ on $\mathcal{G}$ is also free and proper
\cite[Prop.~4.4]{FOR}. Let $(S_{red},P_{red})$ be the quotient poly-Poisson
structure on $M/G$.

\begin{theorem}\label{thm:main}
Let $0\in \LG^*_{(k)}$ be a clean value for the moment map
\eqref{eq:momG}. Then $\mathcal{G}_{red}= J^{-1}(0)/\G$ is a Lie
groupoid over $M/\G$, and the reduced form $\w_{red} \in
\Omega^2(\mathcal{G}_{red},\mathbb{R}^k)$ makes it into a
poly-symplectic groupoid integrating $(S_{red},P_{red})$.
\end{theorem}

\begin{proof}
Let $V_M\subset TM$ be the vertical bundle with respect to the action on $M$. According with
\cite[Lemma 3.1]{BC2} and condition (a) on \eqref{Def:p.reductible} we conclude that
$(J^s)^{-1}(0)=S\cap \opk \Ann(V_M)$ is  Lie subalgebroid of $S$. The $\G$-invariance
allows us to construct, as in \cite[Prop. 4.3]{BC2}, the reduced Lie algebroid
$S_{red}=(J^s)^{-1}(0)/\G$ over $M/\G$.
Furthermore, the reduced Lie algebroid $S_{red}$ coincides the one defined by the reduced
poly-Poisson structure of Theorem \ref{Thm:PPred}.

If $0$ is a clean value for $J$, $J^{-1}(0)$ is  Lie subgroupoid
(see \cite[Lemma 5.1]{BC2}). Following the same lines of \cite[Prop.
5.2]{BC2}, we see that $\mathcal{G}_{red}= J^{-1}(0)/\G$ is a Lie
groupoid over $M/\G$, whose Lie algebroid is $(S_{red},P_{red})$,
and the quotient map $\Pi_0:J^{-1}(0)\to \gpd _{red}$ is a groupoid
morphism.

Let $\w_{red}$ be the reduced form on $\gpd _{red}$, characterized
by $\Pi_0^*\w_{red}=i_0^*\w$, where $i_0$ is the natural inclusion
of $J^{-1}(0)$ on $\gpd$. The second part of \cite[Prop. 5.2]{BC2}
allows us to conclude that $\w_{red}$ is multiplicative.

The fact that the quotient map $\Pi_0$ and the inclusion $i_0$ are groupoid morphism yields
$$
\Pi_0^*(i_{\bar{u}^R}\w_{red})=i_{u^R}\Pi_0^*\w_{red}=i_0^*(i_{u^R}\w)
$$
for any $\bar{u}\in S_{red}$ and $u=d\Pi_{(k)}^*\bar{u}\in S\cap
\opk \Ann(V_M)$, where $\bar{u}^R$ and $u^R$ are the respective
right-invariant vector fields on the correspondent Lie groupoid.
Moreover, if $t,t_0,t_{red}$ denote the target maps on the Lie
groupoids $\gpd$, $J^{-1}(0)$ and $\gpd _{red}$, respectively, we
have
$$
i_0^*(i_{u^R}\w)=i_0^*(t^*u)=t_0^*d\Pi_{(k)}^*\bar{u}=\Pi_0^*t^*_{red}\bar{u},
$$
which implies that $i_{\bar{u}^R}\w_{red}=t_{red}^*\bar{u}$. It
follows from Prop.~\ref{Prop:nondeg. IM-form} that $\w_{red}$ is
nondegenerate, so $(\mathcal{G}_{red},\w_{red})$ is a
poly-symplectic groupoid, and it integrates $(S_{red},P_{red})$.
\end{proof}

Theorem~\ref{thm:main} is a generalization of the following example.

\begin{example}
In Example \ref{Ex:trivialPP} we saw that $\opk
T^*Q\rightrightarrows Q$ is the poly-symplectic Lie groupoid
integrating the trivial $k$-poly-Poisson structure on $Q$. In this
case, for a free and proper $\G$-action on $Q$, the hamiltonian
action of Prop.~\ref{Prp:MMgpd} is the one induced by cotangent
lift, see Example~\ref{ex:ppred}(a). We conclude that the
poly-symplectic reduction in Theorem~\ref{thm:main} is $\opk
T^*(Q/\G)$, as in Example~\ref{ex:redleaf}, which is a presymplectic
groupoid integrating the trivial poly-Poisson structure on $Q/\G$.
\end{example}

\begin{example}
Recall that for a simply connected manifold $M$, the
$k$-poly-symplectic manifold $(M,\w)$, viewed as poly-Poisson
manifold, is integrated by the $s$-simply connected poly-symplectic
groupoid $M\times M \rightrightarrows M$ endowed with the
poly-symplectic form $t^*\w-s^*\w$, where $t,s$ are the natural
projections from $M\times M$ to $M$. If $(M,\w)$ is equipped with a
hamiltonian poly-symplectic action of the Lie group $\G$ and
$J_0:M\to \LG_{(k)}^*$ is its moment map, then the moment map
\eqref{eq:momG} for the hamiltonian action on the groupoid is
$J=t^*J_0-s^*J_0$. If the action on $M$ is free, proper, reducible
and $0\in \LG_{(k)}^*$ is a clean value for $J$, then the symplectic
groupoid $J^{-1}(0)/\G$ over $M/\G$ integrates the reduced
poly-Poisson structure $(S_{red},P_{red})$ induced by $(M,\w)$.
\end{example}


The poly- symplectic groupoid $\gpd _{red}$ in Theorem
\ref{thm:main} is not necessarily the source-simply connected Lie
groupoid integrating the reduced structure. This claim  is
illustrated on \cite[Example 4.8]{FOR} for the case $k=1$.

\begin{remark}
Rather than assuming that $0$ is a clean value of the moment map $J$
on $\mathcal{G}$, one can also proceed as in \cite[Prop.~5.3]{BC2}
and consider the source-simply-connected groupoid $\mathcal{G}_0$
integrating the Lie algebroid $(J^s)^{-1}(0)$. With the same
arguments as in \cite[Prop.~5.3]{BC2}, one can see that this Lie
groupoid is equipped with a $\G$-action and inherits a $\G$-basic
multiplicative 2-form $\omega_0\in
\Omega^2(\mathcal{G}_0,\mathbb{R}^k)$ from the natural map
$\mathcal{G}_0\to \mathcal{G}$, integrating the inclusion
$(J^s)^{-1}(0)\to S$. Then $\mathcal{G}_{0,red}=\mathcal{G}_0/\G$ is a
Lie groupoid over $M/\G$ and $\omega_0$ reduces to a poly-symplectic
form $\omega_{0,red}$ on $\mathcal{G}_{0,red}$ integrating the quotient
poly-Poisson structure $(S_{red},P_{red})$.
\end{remark}

Finally, previous remark allows us to conclude that reduced poly-Poisson structure
$(S_{red},P_{red})$ is integrable if the Lie algebroid $(S,P)$ is also integrable.

\end{document}